\documentclass[preprint,12pt,sort&compress,numbers]{elsarticle}

\usepackage{amsmath,amsthm,amsfonts,amssymb,latexsym,mathrsfs,url}

\newtheorem{thm}{Theorem}[section]
\newtheorem{lem}[thm]{Lemma}
\newtheorem{coro}[thm]{Corollary}

\newtheorem{conj}[thm]{Conjecture}

\theoremstyle{definition}

\newtheorem{exm}[thm]{Example}
\newtheorem{rem}[thm]{Remark}

\def \L{\mathscr{L}}

\journal{Adv. in Appl. Math.}

\begin{document}

\begin{frontmatter}

\title{Log-convex and Stieltjes moment sequences}

\author[a]{Yi Wang\corref{cor1}}
\ead{wangyi@dlut.edu.cn}
\author[b]{Bao-Xuan Zhu\corref{cor2}}
\ead{bxzhu@jsnu.edu.cn}
\cortext[cor2]{Corresponding author.}
\address[a]{School of Mathematical Sciences, Dalian University of Technology, Dalian 116024, PR China}
\address[b]{School of Mathematics and Statistics, Jiangsu Normal University, Xuzhou 221116, PR China}

\begin{abstract}
We show that Stieltjes moment sequences are infinitely log-convex,
which parallels a famous result that (finite) P\'olya frequency
sequences are infinitely log-concave. We introduce the concept of
$q$-Stieltjes moment sequences of polynomials and show that many
well-known polynomials in combinatorics are such sequences. We
provide a criterion for linear transformations and convolutions
preserving Stieltjes moment sequences. Many well-known combinatorial
sequences are shown to be Stieltjes moment sequences in a unified
approach and therefore infinitely log-convex, which in particular
settles a conjecture of Chen and Xia about the infinite
log-convexity of the Schr\"oder numbers. We also list some
interesting problems and conjectures about the log-convexity and the
Stieltjes moment property of the (generalized) Ap\'ery numbers.
\end{abstract}


\end{frontmatter}

\section{Introduction}

Let $\alpha=(a_k)_{k\ge 0}$ be a sequence of nonnegative numbers.
The sequence is called {\it log-convex} ({\it log-concave}, resp.)
if $a_{k}a_{k+2}\ge a_{k+1}^2$ ($a_{k}a_{k+2}\le a_{k+1}^2$, resp.) for all $k\ge 0$.
The log-convex and log-concave sequences arise often in combinatorics
and have been extensively investigated.
We refer the reader to \cite{Bre94,Sta89,WY07} for the log-concavity and \cite{LW07,Zhu13} for the log-convexity.
A basic approach to such problems comes from the theory of total positivity \cite{Bre89,Bre94,Bre95,Bre96,CLW-EuJC15,CLW-LAA15}.

Let $A=[a_{n,k}]_{n,k\ge 0}$ be a finite or infinite matrix of real numbers.
It is called {\it totally positive} ({\it TP} for short) if all its minors are nonnegative.
It is called {\it TP$_2$} if all minors of order $\le 2$ are nonnegative.
Given a sequence $\alpha=(a_k)_{k\ge 0}$,
define its {\it Toeplitz matrix} $T(\alpha)$ and {\it Hankel matrix} $H(\alpha)$ by
$$T(\alpha)=[a_{i-j}]_{i,j\ge 0}=
\left[
  \begin{array}{ccccc}
    a_0 &  &  &  & \\
    a_1 & a_0 &  &  & \\
    a_2 & a_1 & a_0 &  & \\
    a_3 & a_2 & a_1 & a_0 & \\
    \vdots &  & \cdots &  & \ddots \\
  \end{array}
\right]$$
and
$$H(\alpha)=[a_{i+j}]_{i,j\ge 0}=
\left[
  \begin{array}{ccccc}
    a_0 & a_1 & a_2 & a_3 & \cdots \\
    a_1 & a_2 & a_3 & a_4 & \\
    a_2 & a_3 & a_4 & a_5 & \\
    a_3 & a_4 & a_5 & a_6 & \\
    \vdots &  &  &  & \ddots \\
  \end{array}
\right].$$
Clearly, a sequence of positive numbers is log-concave (log-convex, resp.) if and only if
the corresponding Toeplitz matrix (Hankel matrix, resp.) is TP$_2$.

We say that $\alpha$ is a {\it P\'olya frequency sequence} ({\it PF} for short) if its Toeplitz matrix $T(\alpha)$ is TP.
Such sequences have been deeply studied in the theory of total positivity \cite{Kar68} and in combinatorics \cite{Bre89}.
For example, the fundamental representation theorem of Schoenberg and Edrei states that
a sequence $a_0=1, a_1, a_2, \ldots$ of real numbers is PF if and only if its generating function has the form
$$\sum_{k\ge 0}a_kx^k=e^{\gamma z}\frac{\prod_{j\ge 1}(1+\alpha_j z)}{\prod_{j\ge 1}(1-\beta_j z)}$$
in some open disk centered at the origin, where $\alpha_j,\beta_j,\gamma\ge 0$ and $\sum_{j\ge 1}(\alpha_j+\beta_j)<+\infty$
(see \cite[p. 412]{Kar68} for instance).
In particular, a finite sequence of nonnegative numbers is PF if and only if its generating function has only real zeros \cite[p. 399]{Kar68}.

We say that $\alpha=(a_k)_{k\ge 0}$ is a {\it Stieltjes moment}
({\it SM} for short) sequence if its Hankel matrix $H(\alpha)$ is
TP. It is well known that $\alpha$ is a Stieltjes moment sequence if
and only if it has the form
\begin{equation}\label{i-e}
a_k=\int_0^{+\infty}x^kd\mu(x),
\end{equation}
where $\mu$ is a non-negative measure on $[0,+\infty)$
(see \cite[Theorem 4.4]{Pin10} for instance).
Stieltjes moment problem is one of classical moment problems
and arises naturally in many branches of mathematics
\cite{ST43,Wid41}.
It is well known that many counting coefficients form Stieltjes moment sequences, 
including the Bell numbers, the Catalan numbers,
the central binomial coefficients, the central Delannoy numbers,
the factorial numbers, the Schr\"oder numbers.
See \cite{LMW15} for details.

Boros and Moll~\cite[p. 157]{BM04} introduced the concept of the
infinite log-concavity. Given a sequence $\alpha=(a_k)_{k\ge 0}$ of
nonnegative numbers, define a new sequence
$\mathcal{L}(\alpha)=(b_k)_{k\ge 0}$ by $b_0=a_0^2$ and
$b_{k+1}=a^2_{k+1}-a_{k}a_{k+2}$. Then the sequence $\alpha$ is
log-concave if and only if the sequence $\mathcal{L}(\alpha)$ is
nonnegative, i.e., all $b_k$ are nonnegative. Call $\alpha$ {\it
infinitely log-concave} if $\mathcal{L}^i(\alpha)$ is nonnegative
for all $i\geq 1$, where
$\mathcal{L}^i=\mathcal{L}(\mathcal{L}^{i-1})$.
The following result was conjectured
independently by Fisk, Stanley, McNamara and Sagan~\cite{MS10}
and shown by Br\"{a}ndr\'{e}n~\cite{Bra09}.

\begin{thm}\label{Brandenthm}
The operator $\mathcal{L}$ preserves the PF property of finite sequences.
A finite P\'olya frequency sequence is therefore infinitely log-concave.
\end{thm}

Very recently, Chen and Xia~\cite{CX11} introduced the notion of the infinite log-convexity.
Let $\alpha=(a_k)_{k\ge 0}$ be an infinite sequence of positive numbers.
Define a new sequence $\L(\alpha)=(c_k)_{k\ge 0}$ by $c_k=a_{k}a_{k+2}-a^2_{k+1}$.
Then $\alpha$ is log-convex if and only if $\L(\alpha)$ is nonnegative.
Call $\alpha$ {\it $m$-log-convex} if $\L^i(\alpha)$ is nonnegative for all $1\le i\le m$ and
{\it infinitely log-convex} if $\L^i(\alpha)$ is nonnegative for all $i\geq 1$.
Chen and Xia~\cite{CX11} showed that some combinatorial sequences,
including the Ap\'ery numbers and the Schr\"oder numbers, are $2$-log-convex via analytic methods.
Based on numerical evidence they further suggested the infinite log-convexity of these sequences.
However, no non-trivial infinitely log-convex sequences is presented.

In the next section we show that Stieltjes moment sequences are infinitely log-convex,
a parallel result to Theorem \ref{Brandenthm}.
So many famous counting coefficients turn to be infinitely log-convex.
For example, the sequence of the large Schr\"oder numbers is a Stieltjes moment sequence,
and is therefore infinitely log-convex. In \S 3, we introduce the
concept of $q$-Stieltjes moment sequences of polynomials and show
that many well-known polynomials in combinatorics are such
sequences, including the Bell polynomials, the Eulerian polynomials,
the Narayana polynomials (of type B), the $q$-central Delannoy
numbers and the $q$-Schr\"oder numbers. In \S 4, we provide a
criterion for the linear transformations and convolutions preserving
Stieltjes moment sequences. The SM properties of many well-known
combinatorial sequences are easily followed from this viewpoint.
Finally in \S 5, we list some interesting problems and conjectures
about the log-convexity and the SM property of the (generalized)
Ap\'ery numbers.

\section{Infinitely log-convex sequences}

In this section we show that Stieltjes moment sequences are infinitely log-convex.
We need the following classical characterization of Stieltjes moment sequences.

Given a sequence $\alpha=(a_k)_{k\ge 0}$, let $\overline{\alpha}=(a_{k+1})_{k\ge 0}$ denote the shifted sequence of $\alpha$.

\begin{lem}[{\cite[Theorem 1.3]{ST43}}]\label{2pd}
The sequence $\alpha$ is a Stieltjes moment sequence if and only if
both $H(\alpha)$ and $H(\overline{\alpha})$ are positive definite matrices.
\end{lem}

\begin{coro}\label{mainremark}
Let $\alpha$ be a Stieltjes moment sequence.
Then $\L(\alpha)$ is positive.
\end{coro}
\begin{proof}
The positive definiteness of $H(\alpha)$ and $H(\overline{\alpha})$
imply that their leading principal minors of order $2$ are positive:
$$a_{2k}a_{2k+2}-a^2_{2k+1}=\det\left[\begin{array}{ll}
a_{2k}&a_{2k+1}\\
a_{2k+1}&a_{2k+2}
\end{array}\right]>0
$$
and
$$a_{2k-1}a_{2k+1}-a^2_{2k}=\det\left[\begin{array}{ll}
a_{2k-1}&a_{2k}\\
a_{2k}&a_{2k+1}
\end{array}\right]>0.
$$
Thus the sequence $\alpha$ is strictly log-convex:
$$a_{k}a_{k+2}-a^2_{k+1}>0,\qquad k=0,1,2,\ldots.$$
In other words, $\L(\alpha)$ is positive.
\end{proof}

If the Hankel matrix of a sequence is positive definite,
then we say that the sequence is a {\it positive definite sequence},
or a {\it Hamburger moment sequence}.
Such a sequence has the form
$$a_k=\int_{-\infty}^{+\infty}x^kd\mu(x),\quad k=0,1,2,\ldots,$$
where $\mu$ is a positive Borel measure on $(-\infty,+\infty)$ (see \cite{ST43} for instance).
By Lemma \ref{2pd},
a sequence $\alpha$ is a Stieltjes moment sequence if and only if both $\alpha$ and $\overline{\alpha}$ are positive definite sequences.

\begin{thm}\label{mainthm}
The operator $\L$ preserves the SM property.
A Stieltjes moment sequence is therefore infinitely log-convex.
\end{thm}
\begin{proof}
Let $A=[a_{ij}]_{0\le i,j\le n}$ be an $n\times n$ matrix.
The compound matrix $C(A)$ of $A$ is the $\binom{n}{2}\times \binom{n}{2}$ matrix,
whose elements are all minors of order $2$ of $A$,
arranged lexicographically according to the row and column indices of the minors.
The compound operation has the following properties:
\begin{itemize}
  \item [\rm (i)] $C(A^T)=C^T(A)$;
  \item [\rm (ii)] $C(AB)=C(A)C(B)$; and
  \item [\rm (iii)] if $A$ is invertible, then so is $C(A)$.
\end{itemize}
See \cite[p. 1]{Kar68} or \cite[p. 21]{HJ85} for details.
Clearly, if $A$ is a positive definite matrix, then so is $C(A)$.
Indeed, if the matrix $A=P^TP$ is congruent to the identity matrix,
then so is its compound matrix $C(A)=C^T(P)C(P)$.

We first show that the operator $\L$ preserves the positive
definiteness of sequences. Let $\alpha=(a_k)_{k\ge 0}$ be a positive
definite sequence. Then all $H_n(\alpha)=[a_{i+j}]_{0\le i,j\le n}$
are positive definite by the definition. Thus the compound matrix
$C(H_n(\alpha))$ is also positive definite. We need to prove that
all $H_n(\L(\alpha))$ are positive definite. The key observation
behind our proof is that $H_{n-1}(\L(\alpha))$ is a principal
submatrix of $C(H_n(\alpha))$ by symmetry. For example, consider the
case $n=3$. Then $H_3(\alpha), C(H_3(\alpha))$ and $H_2(\L(\alpha))$ are
$$\left[
\begin{array}{rrrr}a_{0}&a_{1}&a_{2}&a_{3}\\
a_{1}&a_{2}&a_{3}&a_{4}\\
a_{2}&a_{3}&a_{4}&a_{5}\\
a_{3}&a_{4}&a_{5}&a_{6}\\
\end{array}\right],$$
$$\left[
\begin{array}{rrrrrr}
\left|\begin{array}{cc}a_{0}&a_{1}\\a_{1}&a_{2}
\end{array}\right|&
\left|\begin{array}{cc}a_{0}&a_{2}\\a_{1}&a_{3}
\end{array}\right|&
\left|\begin{array}{cc}a_{0}&a_{3}\\a_{1}&a_{4}
\end{array}\right|&
\left|\begin{array}{cc}a_{1}&a_{2}\\a_{2}&a_{3}
\end{array}\right|&
\left|\begin{array}{cc}a_{1}&a_{3}\\a_{2}&a_{4}
\end{array}\right|&
\left|\begin{array}{cc}a_{2}&a_{3}\\a_{3}&a_{4}
\end{array}\right|
\\[5mm]
\left|\begin{array}{cc}a_{0}&a_{1}\\a_{2}&a_{3}
\end{array}\right|&
\left|\begin{array}{cc}a_{0}&a_{2}\\a_{2}&a_{4}
\end{array}\right|&
\left|\begin{array}{cc}a_{0}&a_{3}\\a_{2}&a_{5}
\end{array}\right|&
\left|\begin{array}{cc}a_{1}&a_{2}\\a_{3}&a_{4}
\end{array}\right|&
\left|\begin{array}{cc}a_{1}&a_{3}\\a_{3}&a_{5}
\end{array}\right|&
\left|\begin{array}{cc}a_{2}&a_{3}\\a_{4}&a_{5}
\end{array}\right|
\\[5mm]
\left|\begin{array}{cc}a_{0}&a_{1}\\a_{3}&a_{4}
\end{array}\right|&
\left|\begin{array}{cc}a_{0}&a_{2}\\a_{3}&a_{5}
\end{array}\right|&
\left|\begin{array}{cc}a_{0}&a_{3}\\a_{3}&a_{6}
\end{array}\right|&
\left|\begin{array}{cc}a_{1}&a_{2}\\a_{4}&a_{5}
\end{array}\right|&
\left|\begin{array}{cc}a_{1}&a_{3}\\a_{4}&a_{6}
\end{array}\right|&
\left|\begin{array}{cc}a_{2}&a_{3}\\a_{5}&a_{6}
\end{array}\right|
\\[5mm]
\left|\begin{array}{cc}a_{1}&a_{2}\\a_{2}&a_{3}
\end{array}\right|&
\left|\begin{array}{cc}a_{1}&a_{3}\\a_{2}&a_{4}
\end{array}\right|&
\left|\begin{array}{cc}a_{1}&a_{4}\\a_{2}&a_{5}
\end{array}\right|&
\left|\begin{array}{cc}a_{2}&a_{3}\\a_{3}&a_{4}
\end{array}\right|&
\left|\begin{array}{cc}a_{2}&a_{4}\\a_{3}&a_{5}
\end{array}\right|&
\left|\begin{array}{cc}a_{3}&a_{4}\\a_{4}&a_{5}
\end{array}\right|
\\[5mm]
\left|\begin{array}{cc}a_{1}&a_{2}\\a_{3}&a_{4}
\end{array}\right|&
\left|\begin{array}{cc}a_{1}&a_{3}\\a_{3}&a_{5}
\end{array}\right|&
\left|\begin{array}{cc}a_{1}&a_{4}\\a_{2}&a_{5}
\end{array}\right|&
\left|\begin{array}{cc}a_{2}&a_{3}\\a_{3}&a_{4}
\end{array}\right|&
\left|\begin{array}{cc}a_{2}&a_{4}\\a_{3}&a_{5}
\end{array}\right|&
\left|\begin{array}{cc}a_{3}&a_{4}\\a_{5}&a_{6}
\end{array}\right|
\\[5mm]
\left|\begin{array}{cc}a_{2}&a_{3}\\a_{3}&a_{4}
\end{array}\right|&
\left|\begin{array}{cc}a_{2}&a_{4}\\a_{3}&a_{5}
\end{array}\right|&
\left|\begin{array}{cc}a_{2}&a_{5}\\a_{3}&a_{6}
\end{array}\right|&
\left|\begin{array}{cc}a_{3}&a_{4}\\a_{4}&a_{5}
\end{array}\right|&
\left|\begin{array}{cc}a_{3}&a_{5}\\a_{4}&a_{6}
\end{array}\right|&
\left|\begin{array}{cc}a_{4}&a_{5}\\a_{5}&a_{6}
\end{array}\right|
\end{array}\right]$$
and
$$\left[\begin{array}{rrr}
\left|\begin{array}{cc}a_{0}&a_{1}\\a_{1}&a_{2}
\end{array}\right|&
\left|\begin{array}{cc}a_{1}&a_{2}\\a_{2}&a_{3}
\end{array}\right|&
\left|\begin{array}{cc}a_{2}&a_{3}\\a_{3}&a_{4}
\end{array}\right|\\[5mm]
\left|\begin{array}{cc}a_{1}&a_{2}\\a_{2}&a_{3}
\end{array}\right|&
\left|\begin{array}{cc}a_{2}&a_{3}\\a_{3}&a_{4}
\end{array}\right|&
\left|\begin{array}{cc}a_{3}&a_{4}\\a_{4}&a_{5}
\end{array}\right|
\\[5mm]
\left|\begin{array}{cc}a_{2}&a_{3}\\a_{3}&a_{4}
\end{array}\right|&
\left|\begin{array}{cc}a_{3}&a_{4}\\a_{4}&a_{5}
\end{array}\right|&
\left|\begin{array}{cc}a_{4}&a_{5}\\a_{5}&a_{6}
\end{array}\right|
\end{array}\right]$$
respectively.
Thus the positive definiteness of $C(H_n(\alpha))$ implies that of $H_{n-1}(\L(\alpha))$.
In other words, the positive definiteness of $\alpha$ implies that of $\L(\alpha)$, as desired.

Now let $\alpha$ be a Stieltjes moment sequence.
Then $\alpha$ and $\overline{\alpha}$ are positive definite sequences by Lemma \ref{2pd},
so are $\L(\alpha)$ and $\L(\overline{\alpha})$.
Note that $\overline{\L(\alpha)}=\L(\overline{\alpha})$.
Hence $\L(\alpha)$ is also a Stieltjes moment sequence again by Lemma \ref{2pd}.
On the other hand, $\L(\alpha)>0$ by Corollary \ref{mainremark},
the sequence $\alpha$ is therefore infinitely log-convex.
This completes the proof of the theorem.
\end{proof}

Stieltjes moment sequences are much better behaved than infinitely
log-convex sequences and there are various approaches to show that a
sequence is a Stieltjes moment sequence. For example, Liang et al.
\cite{LMW15} showed that many Catalan-like numbers form Stieltjes
moment sequences via the total positivity of the corresponding
Aigner's recursive matrices \cite{Aig99}, including the Bell
numbers, the Catalan numbers, the central binomial coefficients, the
central Delannoy numbers, the factorial numbers and the large
Schr\"oder numbers (see also Corollary \ref{p-sm}). These
Catalan-like numbers are therefore infinitely log-convex, which, in
particular, settles a conjecture of Chen and Xia about the infinite
log-convexity of the Schr\"oder numbers \cite[Conjecture 5.4]{CX11}.

\section{Stieltjes moment sequences of polynomials}

Let $f(q)$ and $g(q)$ be two real polynomials in $q$. We sat that
$f(q)$ is {\it $q$-nonnegative} if $f(q)$ has nonnegative coefficients.
Denote $f(q)\ge_q g(q)$ if $f(q)-g(q)$ is $q$-nonnegative. Let
$A(q)=[a_{n,k}(q)]_{n,k\ge 0}$ be a matrix whose entries are all
real polynomials in $q$. We say that $A(q)$ is {\it $q$-TP} if all
minors are $q$-nononegative. Let $\alpha(q)=(a_n(q))_{n\ge 0}$ be a
sequence of real polynomials in $q$. We say that the sequence is
{\it strongly $q$-log-convex} ($q$-SLCX for short) if
$$a_{n+1}(q)a_{m-1}(q)\ge_q a_{n}(q)a_{m}(q)$$ for $n\ge m\ge 0$. If the
Hankel matrix $H(\alpha(q))=[a_{i+j}(q)]_{i,j\ge 0}$ is $q$-TP, then
we say that $\alpha(q)$ is a {\it $q$-Stieltjes moment} ($q$-SM for
short) sequence. If $\alpha(q)$ is a Stieltjes moment sequence for
any fixed $q\ge 0$, then we say that $\alpha(q)$ is a {\it pointwise
Stieltjes moment} (PSM for short) sequence of polynomials. Clearly,
a $q$-SM sequence is both $q$-SLCX and PSM.

The simplest non-trivial $q$-SM sequence in combinatorics should be
$((q+1)^n)_{n\ge 0}$. Zhu \cite{Zhu13} showed tha the Bell
polynomials, the Eulerian polynomials, the Narayana polynomials (of
type B), the $q$-central Delannoy numbers, the $q$-Schr\"oder
numbers are $q$-SLCX. In this section we show that these polynomials
actually form $q$-SM sequences.

Let $\sigma=(s_k(q))_{k\ge 0}$ and $\tau=(t_{k+1}(q))_{k\ge 0}$ be two sequences of polynomials.
Define an infinite lower triangular matrix $R(q):=R^{\sigma,\tau}(q)=[r_{n,k}(q)]_{n,k\ge 0}$ by the recurrence
\begin{eqnarray}\label{rr-q}
r_{0,0}(q)=1,\quad r_{n+1,k}(q)=r_{n,k-1}(q)+s_k(q)r_{n,k}(q)+t_{k+1}(q)r_{n,k+1}(q),
\end{eqnarray}
where $r_{n,k}(q)=0$ unless $n\ge k\ge 0$.
Similar to Aigner \cite{Aig01},
we say that $R(q)$ is {\it the $q$-recursive matrix} and
$r_{n,0}(q)$ are {\it the $q$-Catalan-like numbers} corresponding to $(\sigma,\tau)$.
Call the tridiagonal matrix (Jacobi matrix)
\begin{equation*}\label{J-eq}
J(q)=\left[
\begin{array}{ccccc}
s_0(q) & 1 &  &  &\\
t_1(q) & s_1(q) & 1 &\\
 & t_2(q) & s_2(q) & 1 &\\
 & & t_3(q) & s_3(q) & \ddots\\
& & &\ddots & \ddots \\
\end{array}\right]
\end{equation*}
{\it the coefficient matrix} of the recurrence \eqref{rr-q}.

\begin{exm}\label{basic-qSM}
Many well-known polynomials are the $q$-Catalan-like numbers.
\begin{itemize}
  \item [\rm (i)] The Bell polynomials $B_n(q)=\sum_{k=0}^nS(n,k)q^k$ when $s_k=k+q$ and $t_{k}=kq$;
  \item [\rm (ii)] The Eulerian polynomials $A_n(q)=\sum_{k=0}^nA(n,k)q^k$ when $s_k(q)=(k+1)q+k$ and $t_k=k^2q$;
  \item [\rm (iii)] The $q$-Schr\"oder numbers $r_n(q)=\sum_{k=0}^n\binom{n+k}{n-k}\frac{1}{k+1}\binom{2k}{k}q^k$ when $s_0=q+1,s_k=2q+1$ and
  $t_k=q(q+1)$;
  \item [\rm (iv)] The $q$-central Delannoy numbers $D_n(q)=\sum_{k=0}^n\binom{n+k}{n-k}\binom{2k}{k}q^k$ when $s_k=1+2q,t_1=2q(1+q)$ and
  $t_k=q(1+q)$;
  \item [\rm (v)] The Narayana polynomials $N_n(q)=\sum_{k=1}^n\frac{1}{n}\binom{n}{k}\binom{n}{k-1}q^k$ when $s_0=q, s_k=1+q$ and
  $t_k=q$;
  \item [\rm (vi)] The Narayana polynomials $W_n(q)=\sum_{k=0}^n\binom{n}{k}^2q^k$ of type B when $s_k=1+q, t_1=2q$ and $t_k=q$ for
  $k>1$.
\end{itemize}
\end{exm}

Liang et al.~\cite[Theorem 2.1]{LMW15} showed that
if the coefficient matrix of a recursive matrix is TP,
then the corresponding Catalan-like numbers form a SM sequence.
The method of the proof used in \cite[Theorem 2.1]{LMW15} can be carried over verbatim to its q-analogue.
Here we omit the details for brevity.

\begin{thm}\label{jh-lem}
If the coefficient matrix $J(q)$ is $q$-TP,
then the corresponding $q$-Catalan-like numbers $r_{n,0}(q)$ form a $q$-SM sequence.
\end{thm}

The following criterion for the total positivity of a tridiagonal matrix will be very useful.

\begin{lem}\label{bc-lem}
Let $b_n(q)$ and $c_n(q)$ be all $q$-nonnegative.
Then the tridiagonal matrix
$$\left[
      \begin{array}{cccc}
        b_1(q)+c_1(q) & 1 &  &  \\
        b_2(q)c_1(q) & b_2(q)+c_2(q) & 1 &  \\
         & b_3(q)c_2(q) & b_3(q)+c_3(q) & \ddots \\
         &  & \ddots & \ddots \\
      \end{array}
    \right]$$
is $q$-TP.
\end{lem}
\begin{proof}
We have the decomposition
\begin{eqnarray*}
&&\left[
      \begin{array}{cccc}
        b_1+c_1 & 1 &  &  \\
        b_2c_1 & b_2+c_2 & 1 &  \\
         & b_3c_2 & b_3+c_3 & \ddots \\
         &  & \ddots & \ddots \\
      \end{array}
    \right]\\
&=&\left[\begin{array}{cccc}
b_1 & 1 &  & \\
 & b_2 & 1 & \\
 &  & b_3 & \ddots\\
 &  &  & \ddots\\
\end{array}\right]
\left[\begin{array}{cccc}
1 &  &  &\\
c_1 & 1 & &\\
 & c_2 & 1 &\\
 & & \ddots & \ddots\\
\end{array}\right].
\end{eqnarray*}
Clearly,
bidiagonal matrices whose entries are $q$-nonnegative are $q$-TP,
and the product of $q$-TP matrices is still $q$-TP.
So the statement follows.
\end{proof}

\begin{coro}\label{ex-qsm}
The six sequences of polynomials in Example \ref{basic-qSM} are all $q$-SM.
\end{coro}
\begin{proof}
It suffices to show that the corresponding coefficient matrices are $q$-TP.
The $q$-total positivity of the first five coefficient matrices may be obtained directly by Lemma \ref{bc-lem}.

(i)\quad For the Bell polynomials, $b_k=k-1$ and $c_k=q$ for $k\ge 1$.

(ii)\quad For the Eulerian polynomials, $b_k=(k-1)q$ and $c_k=k$ for $k\ge 1$.

(iii)\quad For the $q$-Schr\"oder numbers, $b_1=0, b_{k+1}=q$ and $c_k=q+1$ for $k\ge 1$.

(iv)\quad For the $q$-central Delannoy numbers, $b_1=1, b_{k+1}=q+1$ and $c_1=2q, c_{k+1}=q$. 

(v)\quad For the Narayana polynomials, $b_k=1$ and $c_k=q$ for $k\ge 1$.

(vi)\quad For the Narayana polynomials of type B, the corresponding coefficient matrix is
$$J(q)=\left[\begin{array}{ccccc}
q+1 & 1 &  &  &\\
2q & q+1 & 1 &\\
 & q & q+1 & 1 &\\
 & & q & q+1 & \ddots\\
& & &\ddots & \ddots \\
\end{array}\right].$$
We show that $J(q)$ is $q$-TP
by showing that all minors of $J(P)$ are $q$-nonnegative.
Clearly, it suffices to consider the following minors of two kinds:
$$d^{(1)}_n(q)=\left|\begin{array}{ccccc}
q+1 & 1 &  &  &\\
q & q+1 & 1 &\\
 & q & q+1 & \ddots &\\
 & & \ddots & \ddots & 1\\
& & & q & q+1\\
\end{array}\right|$$
and
$$d^{(2)}_n(q)=\left|\begin{array}{ccccc}
q+1 & 1 &  &  &\\
2q & q+1 & 1 &\\
 & q & q+1 & \ddots &\\
 & & \ddots & \ddots & 1\\
& & & q & q+1\\
\end{array}\right|.
$$
By an inductive argument, we obtain $d^{(1)}_n(q)=q^n+\cdots+q+1$
and $d^{(2)}_n(q)=q^n+1$. It is clear that both of them are
$q$-nonnegative, as required. Thus $J(q)$ is $q$-TP.
\end{proof}

\begin{rem}
Given a sequence of polynomials, sometimes we may construct a $q$-recursive matrix
such that these polynomials are precisely the $q$-Catalan-like numbers of this matrix.
As an example, consider the Morgan-Voyce polynomials $M_n(q)=\sum_{k=0}^n\binom{n+k}{n-k}q^k$.
Define a $q$-recursive matrix $M(q)$ by setting $s_0=q+1, s_k=1, t_1=q, t_{k+1}=0$ for $k\ge 1$.
Then it is not difficult to verify that $M_n(q)$ is precisely the corresponding $q$-Catalan-like numbers.
By Lemma \ref{bc-lem},
the coefficient matrix of $M(q)$ can be decomposed into the product of two bidiagonal matrices with $b_k=1, c_1=q$ and $c_{k+1}=0$ for $k\ge 1$,
and is therefore $q$-TP.
Thus $M_n(q)$ form a $q$-SM sequence.
\end{rem}
\begin{coro}\label{p-sm}
The following numbers form Stieltjes moment sequences respectively.
\begin{itemize}
  \item [\rm (i)] The Bell numbers $B_n=\sum_{k=0}^nS(n,k)=B_n(1)$;
  \item [\rm (ii)] The factorial numbers $n!=\sum_{k=0}^nA(n,k)=A_n(1)$;
  \item [\rm (iii)] The Schr\"oder numbers $r_n=r_n(1)$;
  \item [\rm (iv)] The central Delannoy numbers $D_n=D_n(1)$;
  \item [\rm (v)] The Catalan numbers $C_n=\frac{1}{n+1}\binom{2n}{n}=N_n(1)$;
  \item [\rm (vi)] The central binomial coefficients $\binom{2n}{n}=W_n(1)$.
\end{itemize}
\end{coro}

\section{Linear transformations and convolutions}

In this section we present some sufficient conditions for linear
transformations and convolutions that preserve Stieltjes moment
sequences. Similar studies have been carried out for P\'olya
frequency sequences \cite{Bre89}, log-concave sequences \cite{WY07},
log-convex sequences \cite{DP49,LW07}, as well as Stieltjes moment
sequences \cite{Ben11}.

Let $A=[a_{n,k}]_{n,k\ge 0}$ be an infinite nonnegative lower triangular matrix.
Define the $A$-linear transformation
\begin{eqnarray}\label{a-lt}
z_n=\sum_{k=0}^na_{n,k}x_k,\qquad n=0,1,2,\ldots
\end{eqnarray}
and the $A$-convolution
\begin{eqnarray}\label{a-c}
z_n=\sum_{k=0}^{n}a_{nk}x_ky_{n-k},\qquad n=0,1,2,\ldots.
\end{eqnarray}
We say that \eqref{a-lt} preserves the SM property: if $(x_n)_{n\ge
0}$ is a Stieltjes moment sequence, then so is $(z_n)_{n\ge 0}$.
Similarly, we say that \eqref{a-c} preserves the SM property: if
both $(x_n)_{n\ge 0}$ and $(y_n)_{n\ge 0}$ are Stieltjes moment
sequences, then so is $(z_n)_{n\ge 0}$.

The following is a classic characterization of Stieltjes moment
sequences (see, e.g., \cite[Theorem 1.3]{ST43} or \cite[p.
132--135]{Wid41}).

\begin{lem}\label{p-c}
The sequence $\alpha=(a_n)_{n\ge 0}$ is a Stieltjes moment sequence
if and only if
$$\sum_{n=0}^Nc_na_n\ge 0$$ for every polynomial
\begin{eqnarray*}
\sum_{n=0}^Nc_nq^n\ge 0
\end{eqnarray*}
on $[0,\infty)$.
\end{lem}

Our main result in this section is the following.

\begin{thm}\label{thm+conv}
Let $A_n(q)=\sum_{k=0}^{n}a_{n,k}q^k$ be the $n$th row generating function of the triangle $A$.
Assume that $(A_n(q))_{n\ge 0}$ is a Stieltjes moment sequence for any fixed $q\ge 0$.
Then both the $A$-linear transformation \eqref{a-lt} and the $A$-convolution \eqref{a-c} preserve the SM property.
\end{thm}
\begin{proof}
It is known~\cite[Theorem 6]{Ben11} that if the $A$-linear
transformation preserves the SM property, then the same goes for the
$A$-convolution. So it suffices to show that the $A$-linear
transformation preserves the SM property.

Let the polynomial
\begin{eqnarray*}
\sum_{n=0}^Nc_nq^n
\end{eqnarray*}
be nonnegative on $[0,+\infty)$.
By the assumption, the sequence
$$A_n(q)=\sum_{k=0}^{n}a_{n,k}q^k,\quad n=0,1,2,\ldots$$
is a Stieltjes moment sequence for any fixed $q\ge 0$.
Hence by Lemma \ref{p-c}, we have
$$\sum_{n=0}^Nc_n\sum_{k=0}^{n}a_{n,k}q^k\ge 0$$
for $q\ge 0$.
Now let $(x_n)_{n\ge 0}$ be a Stieltjes moment sequence.
Then
$$\sum_{n=0}^Nc_n\sum_{k=0}^{n}a_{n,k}x_k\ge 0$$
by Lemma \ref{p-c}.
Thus the sequence
$$z_n=\sum_{k=0}^{n}a_{n,k}x_k,\quad n=0,1,2,\ldots$$
is a Stieltjes moment sequence again by Lemma \ref{p-c}.
In other words, the linear transformation \eqref{a-lt} preserves the SM property,
as desired.
\end{proof}

The $n$th row generating function of the Pascal triangle is $\sum_{k=0}^n\binom{n}{k}q^k=(q+1)^n$.
So the following corollary is an immediate consequence of Theorem \ref{thm+conv},
which is due to P\'olya and Szeg\"o~\cite[Part VII, Theorem 42]{PS64}

\begin{coro}
If both $(x_n)_{n\ge 0}$ and $(y_n)_{n\ge 0}$ are Stieltjes moment sequences,
then so is their binomial convolution
$$z_n=\sum_{k=0}^{n}\binom{n}{k}x_ky_{n-k},\quad n=0,1,2,\ldots.$$
\end{coro}

Since a $q$-Stieltjes moment sequence of polynomials is a pointwise Stieltjes moment sequence,
we have the following result. 

\begin{coro}\label{conv}
The following convolutions preserve the SM properties:
\begin{itemize}
  \item [\rm (i)] $z_n=\sum_{k=0}^nS(n,k)x_ky_{n-k}$;
  \item [\rm (ii)] $z_n=\sum_{k=0}^nA(n,k)x_ky_{n-k}$;
  \item [\rm (iii)] $z_n=\sum_{k=0}^n\frac{1}{n}\binom{n}{k}\binom{n}{k-1}x_ky_{n-k}$;
  \item [\rm (iv)] $z_n=\sum_{k=0}^n\binom{n}{k}^2x_ky_{n-k}$;
  \item [\rm (v)] $z_n=\sum_{k=0}^n\binom{n+k}{n-k}x_ky_{n-k}$.
\end{itemize}
\end{coro}

\begin{rem}
Taking $x_k\equiv 1$ and $y_k\equiv 1$ in Corollary \ref{conv}
(i)-(iv), we obtain the SM property of the Bell numbers $B_n$, the
factorial numbers $n!$, the Catalan numbers $C_n$, and the central
binomial coefficients $b(n)=\binom{2n}{n}$ again. Taking $y_k\equiv
1, x_k=b(k),C_k$ and $1$ in (v), we obtain the SM property of the
central Delannoy numbers $D_n$, the Schr\"oder numbers $r_n$, and
the Fibonacci numbers $F_{2n+1}$ with odd indices respectively.
\end{rem}

\section{Conjectures and open problems}

The Ap\'ery numbers
$$A_n=\sum_{k=0}^n\binom{n}{k}^2\binom{n+k}{k}^2,\quad n=0,1,2,\ldots$$
and
$$B_n=\sum_{k=0}^n\binom{n}{k}^2\binom{n+k}{k},\quad n=0,1,2,\ldots$$
were introduced by Ap\'ery in his famous proof to the irrationality of $\zeta(3)=\sum_{n\ge 1}n^{-3}$.
Chen and Xia \cite{CX11} proposed the following conjecture.

\begin{conj}
Both $(A_n)_{n\ge 0}$ and $(B_n)_{n\ge 0}$ are infinitely log-convex sequences.
\end{conj}

Independently, Sokal \cite{Sok} suggested the following conjecture.

\begin{conj}
Both $(A_n)_{n\ge 0}$ and $(B_n)_{n\ge 0}$ are Stieltjes moment sequences.
\end{conj}

Let $r$ and $s$ be two positive integers.
Define the generalized Ap\'ery polynomials
$$A_n(r,s;q)=\sum_{k=0}^n\binom{n}{k}^r\binom{n+k}{k}^sq^k,\quad n=0,1,2,\ldots.$$
A general problem to ask is whether $A_n(r,s;1)$ form an infinitely
log-convex sequence or even a Stieltjes moment sequence. We have
known that the $q$-central Delannoy numbers $D_n(q)=A_n(1,1;q)$ are
$q$-SM. Sun~\cite{Sun1102} conjectured that $A_n(2,1;q)$ are
$q$-log-convex. In which case, $A_n(r,s;q)$ are PSM, $q$-SLCX, or
even $q$-SM?

\section*{Acknowledgements}

This work was supported in part by the National Natural Science
Foundation of China (Nos. 11201191, 11371078, 11571150).
The authors thank the anonymous referee for his/her helpful comments.

\end{document}